\documentclass[12pt,reqno]{amsart}

\usepackage{amssymb}
\usepackage[all]{xy}

\usepackage{hyperref}
\hypersetup{colorlinks,
linkcolor=black,
filecolor=black,
urlcolor=black,
citecolor=black}

\numberwithin{equation}{section}

\newtheorem{theorem}{Theorem}[section]
\newtheorem{proposition}[theorem]{Proposition}
\newtheorem{lemma}[theorem]{Lemma}

\theoremstyle{definition}
\newtheorem{definition}[theorem]{Definition}
\newtheorem{remark}[theorem]{Remark}

\begin{document}

\baselineskip=15pt

\title[Parabolic bundles and Lie algebroid connections]{Parabolic vector bundles and Lie 
algebroid connections}

\author[D. Alfaya]{David Alfaya}

\address{Department of Applied Mathematics and Institute for Research in Technology, ICAI 
School of Engineering, Comillas Pontifical University, C/Alberto Aguilera 25, 28015 Madrid, 
Spain}

\email{dalfaya@comillas.edu}

\author[I. Biswas]{Indranil Biswas}

\address{Department of Mathematics, Shiv Nadar University, NH91, Tehsil Dadri,
Greater Noida, Uttar Pradesh 201314, India}

\email{indranil.biswas@snu.edu.in, indranil29@gmail.com}

\author[P. Kumar]{Pradip Kumar}

\address{Department of Mathematics, Shiv Nadar University, NH91, Tehsil Dadri,
Greater Noida, Uttar Pradesh 201314, India}

\email{Pradip.Kumar@snu.edu.in}

\author[A. Singh]{Anoop Singh}

\address{Department of Mathematical Sciences, Indian Institute of Technology (BHU), Varanasi 
221005, India}

\email{anoopsingh.mat@iitbhu.ac.in}

\subjclass[2010]{14H60, 53B15, 70G45}

\keywords{Parabolic bundle, Lie algebroid connection, Atiyah exact sequence, stable bundle}

\date{}

\begin{abstract}
Given a holomorphic Lie algebroid on an $m$--pointed connected Riemann surface, we define 
parabolic Lie algebroid connections on any parabolic vector bundle equipped with parabolic 
structure over the marked points. An analogue of the Atiyah exact sequence for parabolic Lie 
algebroids is constructed. For any Lie algebroid whose underlying holomorphic vector bundle is 
stable, we give a complete characterization of all the parabolic vector bundles that admit a 
parabolic Lie algebroid connection.
\end{abstract}

\maketitle

\section{Introduction}

Let $X$ be a compact connected Riemann surface and $E$ a holomorphic vector bundle over $X$. A 
Lie algebroid on $X$ is a locally free coherent analytic sheaf $V$ on $X$ equipped with a Lie 
algebra structure together with a holomorphic homomorphism of vector bundles $\phi\, :\, V\, 
\longrightarrow\, TX$ that intertwines the Lie algebra structure with the Lie bracket 
operation on $TX$. The notion of holomorphic connections on a holomorphic vector bundle $E$ on 
$X$ extends to holomorphic Lie algebroid connections on $E$. In the definition of holomorphic 
Lie algebroid connections $TX$ is replaced by $V$; the homomorphism $\phi$ is used in 
formulating the Leibniz identity.

Lie algebroid connections are quite similar to Simpson's notion of $\Lambda$--modules 
\cite{Si3}, \cite{To1}, \cite{To2}. In fact, Lie algebroid connections simultaneously 
generalize a number of geometric structures that appear in differential geometry, algebraic 
geometry and foliations. It also appears in mathematical physics \cite{CM}, \cite{LM}. Some 
examples of topics where Lie algebroid connections appear: Higgs bundles \cite{Hi},
\cite{Si2}, \cite{Si3}; twisted Higgs bundles \cite{Ni}; flat connections \cite{Si2}; 
logarithmic or meromorphic connections \cite{De}, \cite{Ni2}, \cite{Bo}, \cite{BS};
foliations \cite{ELW}, \cite{PW}, \cite{PP} (see also the references in these works).

In \cite{BKS}, a criterion for the existence of a Lie algebroid connection on $E$ was given 
under the assumption that the vector bundle $V$ is stable.

Parabolic vector bundles are generalizations of holomorphic vector bundles.
Parabolic vector bundles on Riemann surfaces were introduced by
Mehta and Seshadri in \cite{MS}, and their definition was further developed by Maruyama and
Yokogawa to higher dimensional projective varieties \cite{MY}. In \cite{MS} it was shown that
the polystable parabolic vector bundles of rank $r$ and parabolic degree zero,
with parabolic structure over $S$, on a compact Riemann surface $X$ correspond to the
representations of the fundamental group of the complement $X\setminus S$ into
${\rm U}(r)$. More generally, polystable parabolic Higgs bundles of rank $r$ and
parabolic degree zero on $X$, with parabolic structure over $S$, correspond to the
completely reducible representations of the fundamental group of the complement
$X\setminus S$ into $\text{GL}(r,{\mathbb C})$ \cite{Si1}.

In this context, it is natural to ask whether Lie algebroid connections extend to the
set-up of parabolic vector bundles. Our aim here is to address this question.

Let $S \,=\, \{x_1,\, \cdots,\, x_m\}\, \subset\, X$ be a nonempty finite subset, whose 
elements will be called the parabolic points. Let $E_*$ be a parabolic vector bundle over $X$ 
with parabolic structure over $S$ (see Section \ref{Para} for the definition). A connection on 
a parabolic vector bundle $E_*$ is a logarithmic connection on the underlying holomorphic 
vector bundle $E$ which is singular over $S$ such that the residue of the connection over any 
$x\, \in\, S$ is compatible with the parabolic structure of $E_*$ over $x$.

Every parabolic vector bundle can be expressed as a direct sum of indecomposable parabolic
vector bundles. A parabolic vector bundle $E_*$ admits a connection if and only if the parabolic degree
of every parabolic vector bundle which is a direct summand of $E_*$ is zero \cite[Theorem 1.1]{BL}.

Given a Lie algebroid $(V,\, \phi)$ and a parabolic vector
bundle $E_*$, we define parabolic Lie algebroid connections on $E_*$ (see Definition
\ref{def:paraLieconn}). If $V\,=\, TX\otimes {\mathcal O}_X(-S)$ and $\phi$ is the natural
inclusion map of $TX\otimes {\mathcal O}_X(-S)$ in $TX$, then a parabolic Lie algebroid connection
on $E_*$ is a usual connection on the parabolic vector bundle.
After having the definition of a parabolic Lie algebroid connection we address the question of
giving a criterion for the existence of a parabolic Lie algebroid connection on a parabolic vector bundle.

As mentioned before, the definition of a connection on a parabolic vector bundle $E_*$
uses the residue, over the points of $S$, of a logarithmic connection. Given a Lie algebroid
$(V,\, \phi)$, for any Lie algebroid connection $D$ on a holomorphic vector bundle $F$ on $X$,
we construct a $\mathbb C$--linear homomorphism
$$
{\mathcal S}_x\, :\, F_x\, \longrightarrow\, F_x\otimes {\rm cokerel}(\phi)_x\,=\,
F_x\otimes {\mathcal Q}_x
$$
for every $x\, \in\, X$ (see \eqref{e4}). This homomorphism is a generalization of the notion
of residue of a logarithmic connection.

Before defining a parabolic Lie algebroid connection, we introduce the weaker notion
of a quasi-parabolic Lie algebroid connection.
We prove the following (see Proposition \ref{prop1}):

\begin{proposition}\label{prop0.1}
Let $(V,\, \phi)$ be a Lie algebroid and $E_*$ a parabolic vector bundle.
A Lie algebroid connection $D\,:\, E\,\longrightarrow \, E\otimes V^*$ on the vector bundle $E$
underlying $E_*$ gives a quasi-parabolic Lie algebroid connection on the
parabolic vector bundle $E_*$ if and only if the following two statements hold:
\begin{enumerate}
\item For every $x\, \in\, S$, the homomorphism ${\mathcal S}_x$ preserves
the quasi-parabolic filtration of $E_*$ over $x$.

\item For all $x\, \in\, S$, the homomorphism of fibers $\phi^*_x\, :\, (K_X)_x\, \longrightarrow\, V^*_x$
is the zero map.
\end{enumerate}
\end{proposition}

A parabolic Lie algebroid connection on $E_*$ is a quasi-parabolic Lie algebroid connection on
$E_*$ such that the action of ${\mathcal S}_x$ on each graded piece of the quasi-parabolic filtration
of $E_*$ over $x$ is given by the parabolic weight of that graded piece; see Definition \ref{def:paraLieconn}.

Next, we develop a parabolic analog of the Atiyah exact sequence
in the set up of Lie algebroids. First we introduce the notion of generalized parabolic Lie
algebroid connections, and the sheaf $\mathcal{C}_{E_*,V}$ of generalized parabolic Lie algebroid 
connections on $E_*$, which is a holomorphic vector bundle over $X$, fits into a short exact sequence (see 
\eqref{Atiya_para_Lie}). A parabolic Lie algebroid connection on $E_*$ is precisely a holomorphic 
splitting of \eqref{Atiya_para_Lie}.

Let $(V,\, \phi)$ be a Lie algebroid such that the homomorphism $\phi$ vanishes over $S$. Then the
dual homomorphism $\phi^*\, :\, K_X\, \longrightarrow\,V^*$ also vanishes over $S$, and hence it
produces a homomorphism
$$
\widetilde{\phi}^* \,\, : \,\, K_X \otimes {\mathcal O}_X(S) \,\,\longrightarrow\,\, V^*.
$$
In Section \ref{Crit_para_Lie_conn}, we give a criterion for the existence of parabolic Lie algebroid 
connections on a parabolic vector bundle. The following is proved (see Theorem \ref{thm1}):

\begin{theorem}\label{thm0.1}
Let $(V,\, \phi)$ be a Lie algebroid such that $V$ is stable and the homomorphism
$\phi$ vanishes over $S$. Let $E_*$ be a parabolic vector bundle on $X$. Then the following two
statements hold:
\begin{enumerate}
\item If the above homomorphism $\widetilde{\phi}^*$ is not an isomorphism, then $E_*$ admits a parabolic
Lie algebroid connection.

\item If $\widetilde{\phi}^*$ is an isomorphism, then $E_*$ admits a parabolic Lie algebroid connection
if and only if the parabolic degree of every parabolic direct summand of $E_*$ is zero.
\end{enumerate}
\end{theorem}

Finally, in the last section, we study integrable Lie algebroid connections and their moduli by analyzing them in the 
framework of parabolic $\Lambda$--modules in the sense of 
\cite{Si3} and \cite{Al}. It is shown that the integrable quasi-parabolic Lie algebroid connections are equivalent to 
the parabolic $\Lambda$-module structures in the sense of Definition \ref{def:parabolicLambdaModule} (\cite[Definition 2.5]{Al}).
Defining the semistable parabolic Lie algebroid connections in this context, we prove the following theorem (see Theorem 
\ref{thm2}).

\begin{theorem}
\label{thm0.2}
Let $(V,\,\phi)$ be any Lie algebroid such that $\phi|_S\,=\,0$. For every system of
weights $\alpha\,=\,\{\{\alpha_i^x\}_{i=1}^{l_x}\}_{x\in S}$, every parabolic type
$\overline{r}\,=\,\{\{r_i^x\}_{i=1}^{l_x}\}_{x\in S}$ and every degree $d$, there exists a quasi-projective
coarse moduli space $\mathcal{M}_{(V,\phi)}(\alpha,\,\overline{r},\,d)$ of semistable integrable parabolic
$(V,\,\phi)$--connections $(E_*,\,D)$ on $(X,\,S)$, with $D\,:\,E\,\longrightarrow\, E\otimes V^*$,
$\deg(E)\,=\,d$, $\dim(E_x^i/E_x^{i+1})\,=\,r_x^i$ and system of weights $\alpha$.

If $V$ is a line bundle, then this moduli space is nonempty if and only if either
\begin{itemize}
\item $(V,\,\phi)\,\,\ne \,\, (TX(-S),\ i\,:\,TX(-S)\,\hookrightarrow \, TX)$, or

\item $(V,\,\phi)\,\,=\,\,(T(-S),\ i\,:\,TX(-S)\,\hookrightarrow\, TX)$ and
$$d+\sum_{x\in S} \sum_{i=1}^{l_x} \alpha_i^x r_i^x \,\,=\,\, 0.$$
\end{itemize}
If $V$ is a stable bundle, then for the moduli space to be nonempty it is sufficient to have either
\begin{itemize}
\item $\operatorname{im} \phi \ne TX(-S)$, or
\item $\operatorname{im} \phi = TX(-S)$ and 
$$d+\sum_{x\in S} \sum_{i=1}^{l_x} \alpha_i^x r_i^x \,=\, 0.$$
\end{itemize}
\end{theorem}

\section{Lie algebroid connections and parabolic bundles}
\label{Lie_conn}

\subsection{Lie algebroid connection}

Let $X$ be a compact connected Riemann surface. The holomorphic cotangent and tangent bundles
of $X$ will be denoted by $K_X$ and $TX$ respectively. The first
holomorphic jet bundle of a holomorphic vector bundle
$W$ on $X$ will be denoted by $J^1(W)$; so $J^1(W)$ is a holomorphic vector bundle on $X$
that fits in the following short exact sequence of holomorphic vector bundles on $X$:
$$
0\, \longrightarrow\, W\otimes K_X \, \longrightarrow\, J^1(W) \, \longrightarrow\, W
\, \longrightarrow\, 0.
$$

A $\mathbb{C}$--Lie algebra structure on a holomorphic vector
bundle $V$ on $X$ is a $\mathbb{C}$--bilinear pairing defined by a sheaf homomorphism
$$
[-,\, -] \,\,:\,\, V\otimes_{\mathbb C} V \,\, \longrightarrow\,\, V,
$$
which is given by an ${\mathcal O}_X$--linear holomorphic
homomorphism $$J^1(V)\otimes J^1(V)\, \longrightarrow\, V$$ of vector bundles, such
that $$[s,\, t]\,=\, -[t,\, s]\ \ \text{ and }\ \ [[s,\, t],\, u]+[[t,\, u],\, s]+[[u,\, s],\, t]\,=\,0$$
for all locally defined holomorphic sections $s,\, t,\, u$ of $V$. The Lie bracket operation on $TX$ gives
the structure of a Lie algebra on it.

A \textit{Lie algebroid} on $X$ is a pair $(V,\, \phi)$, where
\begin{enumerate}
\item $V$ is a holomorphic vector bundle on $X$ equipped with the structure of a $\mathbb{C}$--Lie algebra,

\item $\phi\, :\, V\, \longrightarrow\, TX$ is an ${\mathcal O}_X$--linear homomorphism such that
$$
\phi([s,\, t])\,=\, [\phi(s),\, \phi(t)]
$$
for all locally defined holomorphic sections $s,\, t$ of $V$, and

\item $[s,\, f\cdot t]\,=\, f\cdot [s,\, t]+\phi(s)(f)\cdot t$ for all locally defined holomorphic sections
$s,\, t$ of $V$ and all locally defined holomorphic functions $f$ on $X$.
\end{enumerate}
The above homomorphism $\phi$ is called the \textit{anchor map} of the Lie algebroid.

\begin{remark}
Observe that the above condition (2) is actually a consequence of conditions (1) and (3), since, for
all holomorphic local sections $s,t,u$ of $V$ and each locally defined
holomorphic function $f$ in $\mathcal{O}_X$ we have, on one hand,
$$[[s,\,t],\,fu]\,=\,f[[s,\,t],\,u]+\phi([s,\,t])(f) u$$
and, on the other hand
$$
[[s,\,t],\,fu]\,=\,[[s,\,fu],\,t]+[s,\,[t,\,fu]]\,=\,[f[s,\,u]+\phi(s)(f)u,\,t]
$$
$$
+[s,\,f[t,\,u]+\phi(t)(f) u] \,=\, f[[s,\,u],\,t]-\phi(t)(f)[s,\,u]+\phi(s)(f)[u,\,t]
$$
$$
-\phi(t)(\phi(s)(f)) u + f[s,\,[t,\,u]]+\phi(s)(f)[t,\,u] +\phi(t)(f)[s,\,u]+\phi(s)(\phi(t)(f))u
$$
$$
=\, f[[s,\,t],\,u]+\left( \phi(s)(\phi(t)(f))-\phi(t)(\phi(s)(f))\right) u\nonumber
$$
$$
=\, f[[s,\,t],\,u]+[\phi(s),\, \phi(t)](f) u.\nonumber
$$
As these coincide for each local holomorphic section $u$ of $V$ and each local holomorphic
function $f$, we must have
$$\phi([s,\,t])\,\,=\,\,[\phi(s),\,\phi(t)].$$
\end{remark}

Let $(V,\, \phi)$ be a Lie algebroid on $X$. We have the dual homomorphism
\begin{equation}\label{e2}
\phi^*\,:\, K_X\, \longrightarrow\, V^*
\end{equation}
of $\phi$. Let
\begin{equation}\label{e3}
q\,\, :\,\, V^* \,\, \longrightarrow\,\, V^*/\phi^*(K_X)\,\,=:\,\, {\mathcal Q}
\end{equation}
be the corresponding quotient map. Note that the above coherent analytic sheaf ${\mathcal Q}$ may
have torsion. For any point $y\, \in\, X$, the quotient ${\mathcal Q}/\mathbf{m}_y {\mathcal Q}$, where
$\mathbf{m}_y\, \subset\, {\mathcal O}_X$
is the maximal ideal associated to $y$, will be denoted by ${\mathcal Q}_y$.

The fiber over any point $y\,\in\, X$ of any vector bundle $W$ on $X$ will be denoted by $W_y$.

A Lie algebroid connection on a holomorphic vector bundle
$E$ on $X$ is a first order holomorphic differential operator
$$
D\,\,:\,\, E\,\,\longrightarrow\, \, E\otimes V^*
$$
such that
\begin{equation}\label{e-4}
D(fs) \,=\, fD(s) + s\otimes \phi^*(df)
\end{equation}
for all locally defined holomorphic section $s$ of $E$ and all locally defined holomorphic
function $f$ on $X$, where $\phi^*$ is the homomorphism constructed in \eqref{e2}.

Consider a Lie algebroid connection $D\,:\, E\,\longrightarrow\, E\otimes V^*$
on $E$. Fix a point $x\, \in\, X$.
Take any $v\, \in\, E_x$. Choose a holomorphic section $s$ of $E$, defined on an open neighborhood $U$ of $x$, such that
$s(x)\,=\, v$. Consider $D(s)\, \in\, H^0\left(U,\, (E\otimes V^*)\big\vert_U\right)$. Let
\begin{equation}\label{ws}
\widehat{s} \,:=\, ({\rm Id}_E\times q)(D(s))(x) \, \in\, E_x\otimes {\mathcal Q}_x
\end{equation}
be the image, where $q$ is the projection in \eqref{e3}. If $s_1$ is another holomorphic section of
$E$, defined on an open neighborhood of $x$, with $s_1(x) \,=\, v$, then
\begin{equation}\label{e3a}
s-s_1\,\,=\,\, f\cdot t,
\end{equation}
where $t$ is a holomorphic section of $E$ defined on an
open neighborhood of $x$, and $f$ is a holomorphic function defined around $x$ with $f(x)\,=\, 0$. From \eqref{e-4}
and \eqref{e3a} it follows that
\begin{equation}\label{e3b}
(D(s) -D(s_1))(x) \,\,=\,\, t(x)\otimes \phi^*(df)(x) + f(x)\cdot (D(t))(x)
\,\,=\,\,t(x)\otimes \phi^*(df)(x),
\end{equation}
because $f(x)\,=\, 0$. Since $q(\phi^*(df))\,=\, 0$, where $q$ is the projection in \eqref{e3}, from
\eqref{e3b} it follows that $$\widehat{s} \,=\, ({\rm Id}_E\times q)(D(s))(x) \,=\, ({\rm Id}_E\times q)(D(s_1))(x).$$
Consequently, we get a linear map
\begin{equation}\label{e4}
{\mathcal S}_x\,\,:\,\, E_x\,\, \longrightarrow\,\, E_x\otimes {\mathcal Q}_x
\end{equation}
that sends any $v\, \in\, E_x$ to $\widehat{s}\,:=\, ({\rm Id}_E\times q)(D(s))(x)$, where $s$ is any holomorphic section
of $E$, defined around $x$, with $s(x)\,=\, v$.

Notice that this linear map $\mathcal{S}_x$ can be defined for any Lie algebroid connection and over each 
point of the curve $x\,\in \,X$; not necessarily a parabolic point in $S$ (parabolic
points are introduced in Section \ref{Para}). Let $S\, \subset\, X$ be a reduced
effective divisor on $X$. In the special case of logarithmic 
connections --- where $(V,\, \phi)\,=\,(TX(-S),\, i:TX(-S)\hookrightarrow TX)$ ---
we have $\phi|_x\, =\, 0$ for $x\, \in\, S$, so $\mathcal{Q}_x\,=\,(K_X\otimes {\mathcal O}_X(x))|_x
\,= \, \mathbb{C}$ (Poincar\'e adjunction formula) and the map $\mathcal{S}_x$ becomes the usual notion 
of residue of a logarithmic connection. Thus, we can understand $\mathcal{S}_x$ as a natural generalization 
for arbitrary Lie algebroids --- and arbitrary points --- of the residue map. In fact, in the next section, we will 
see that, when $x\in S$ is a parabolic point, the map $\mathcal{S}_x$ plays exactly the same role as the 
residue in the characterization of parabolic Lie algebroid connections.

\subsection{Parabolic bundles}\label{Para}

Let $S \,=\, \{x_1,\, \cdots,\, x_m\}\, \subset\, X$ be a nonempty
finite subset, whose elements will be called the parabolic points. Let 
$E$ be a holomorphic vector bundle over $X$. A \emph{quasi-parabolic structure} on $E$ over $x\, \in\, S$ is a
strictly decreasing filtration of subspaces of the fiber $E_x$
\begin{equation}\label{eq:a1}
E_x \,=\, E^1_x \,\supsetneq\, E^2_x \,\supsetneq\, \, \cdots\, \supsetneq\, E^{\ell_x}_x\,\supsetneq\, E^{\ell_x+1}_x \,=\, 0.
\end{equation}
A parabolic structure on $E$ over $x\, \in\, S$ is a quasi-parabolic filtration as in \eqref{eq:a1}
together with a sequence of real numbers
\begin{equation}\label{ew}
0 \,\leq\, \alpha^x_1 \,< \,\cdots\, <\, \alpha^x_{\ell_x} \,< \,1.
\end{equation}
Note that $\ell_x\, \geq\, 1$ if $E\, \not=\, 0$.
A parabolic structure on $E$ is parabolic structure on $E$ on every $x\, \in\, S$ (see \cite{MS}, \cite{MY}). A parabolic
vector bundle on $X$ consists of a vector bundle $E$ on $X$ together with a parabolic structure on $E$. For notational
convenience, such a parabolic vector bundle will also be denoted by $E_*$.

Take a parabolic vector bundle
\begin{equation}\label{e6}
E_*\,=\, (E,\, \{\{E^i_x\}_{i=1}^{\ell_x}\}_{x\in S},\, \{\{\alpha^x_i\}_{i=1}^{\ell_x}\}_{x\in S})
\end{equation}
as above. For any $x\,\in\, S$ and any $1\,\leq\, i\,\leq\,\ell_x+1$, let ${\mathcal E}_{x,i}$ be the unique holomorphic vector bundle
on $X$ that fits in the following short exact sequence of coherent analytic sheaves on $X$:
\begin{equation}\label{e5}
0\, \longrightarrow\, {\mathcal E}_{x,i}\, \xrightarrow{\,\,\,\iota_{x,i}\,\,}\, E \,
\longrightarrow\, E_x/E^i_x\, \longrightarrow\, 0.
\end{equation}
So ${\mathcal E}_{x,i}$ is identified with $E$ over the complement $X\setminus\{x\}$. These $\{{\mathcal E}_{x,i}\}$
form a decreasing sequence of subsheaves of $E$ for each $x\,\in \,S$:
\begin{equation}\label{eq:parSubsheaves}
E\,=\,{\mathcal E}_{x,1}\,\supsetneq\, {\mathcal E}_{x,2} \,\supsetneq\,\cdots
\,\supsetneq\, {\mathcal E}_{x,l_x} \,\supsetneq \,{\mathcal E}_{x,l_x+1}\,=\,E\otimes {\mathcal O}_X(-x).
\end{equation}

Consider the endomorphism bundle $\text{End}(E)\,=\, E\otimes E^*$. It has a coherent analytic
subsheaf
\begin{equation}\label{pe}
\text{End}_P(E)\, \subset \, \text{End}(E)
\end{equation}
defined by the following condition: For any
open subset $U\, \subset\, X$, any $$s\,\, \in\,\, H^0(U,\, \text{End}(E)\big\vert_U)$$ lies
in $H^0(U,\, \text{End}_P(E)\big\vert_U)$ if and only if 
$s(x)(E^i_x)\, \subset\, E^i_x$ for all $x\, \in\, S\bigcap U$ and $1\, \leq\, i\, \leq\, \ell_x$. Let
\begin{equation}\label{pe2}
\text{End}_n(E)\, \subset \, \text{End}_P(E)
\end{equation}
be the subsheaf defined by all $s$ as above such that $s(x)(E^i_x)\, \subset\, E^{i+1}_x$
for all $x\, \in\, S\bigcap U$ and $1\, \leq\, i\, \leq\, \ell_x$.

The notions of semistable, stable and polystable vector bundles generalize to the parabolic context
(see \cite{MS}, \cite{MY}). A parabolic vector bundle $E_*$ is called stable (respectively, semistable) if
for any nontrivial subbundle $F\subset E$ we have the following inequality between the parabolic slopes
of $E_*$ and the induced parabolic structure on $F$ by $E_*$:
$$
\frac{\deg(F)+\sum_{x\in S}\sum_{i=1}^{l_x} \alpha_i^x (\dim(E_x^i\cap F_x)-\dim(E_x^{i+1}\cap F_x))}{\operatorname{rk}(F)}
$$
$$
<\, \,\,\, \text{(respectively, }\le \text{)}\,\,\,
\frac{\deg(E)+\sum_{x\in S}\sum_{i=1}^{l_x} \alpha_i^x (\dim(E_x^i)-\dim(E_x^{i+1}))}{\operatorname{rk}(E)}.
$$
We call the quotient
$$\text{par-}\mu(E_*) \,=\, \frac{\deg(E)+\sum_{x\in S}\sum_{i=1}^{l_x} \alpha_i^x (\dim(E_x^i)-\dim(E_x^{i+1}))}{\operatorname{rk}(E)}$$
the parabolic slope of $E_*$. A parabolic vector bundle $E_*$ is called polystable if it is a direct sum of
stable parabolic bundles of same parabolic slope. Also, the notion of Harder--Narasimhan
filtration of vector bundles extends to the parabolic context (see \cite{MY}).

\section{Connections on parabolic bundles}\label{conn_para}

\subsection{Quasi-parabolic Lie algebroid connections}

Let $(V,\, \phi)$ be a Lie algebroid on $X$. Take a parabolic vector bundle $E_*$ on $X$ as in \eqref{e6}.

\begin{definition}
\label{def:parabolicConnection}
A \textit{quasi-parabolic Lie algebroid connection} on $E_*$ is a Lie algebroid connection
$$
D\,\,:\,\, E\,\,\longrightarrow\, \, E\otimes V^*
$$
on $E$ (see \eqref{e-4}) such that
\begin{equation}\label{e7}
D({\mathcal E}_{x,i}) \,\, \subset\, \, {\mathcal E}_{x,i}\otimes V^*
\end{equation}
for all $x\, \in\, S$ and every $1\, \leq\, i\, \leq\, \ell_x+1$ (see \eqref{e5}).
\end{definition}

\begin{lemma}\label{lem1}
Fix $x\, \in\, S$ and $1\, \leq\, i\, \leq\, \ell_x$. Take a Lie algebroid connection
$$D\ :\ E\ \longrightarrow \
E\otimes V^*$$ such that the homomorphism ${\mathcal S}_x$ in \eqref{e4} maps $E^i_x\, \subset\, E_x$ to
$E^i_x\otimes {\mathcal Q}_x$. Then $D$ produces a homomorphism
$$
{\mathbb D}_{x,i}
 \,\, :\,\, {\mathcal E}_{x,i}\,\, \longrightarrow\,\, (E_x/E^i_x)\otimes\phi^*((K_X)_x)
\,\,\subset\,\, (E_x/E^i_x)\otimes V^*_x
$$
which satisfies the identity ${\mathbb D}_{x,i}(f\cdot s) \,=\, f(x)\cdot {\mathbb D}_{x,i}(s)$ for
every holomorphic section $s$ of ${\mathcal E}_{x,i}$ defined around $x\, \in\, S$ and every holomorphic
function $f$ defined on an open neighborhood of $x$, where $\phi^*$ is the homomorphism in \eqref{e2}.
\end{lemma}

\begin{proof}
We have the following diagram of homomorphisms
\begin{equation}\label{e9}
\begin{matrix}
&& {\mathcal E}_{x,i}\\
&& \,\,\,\,\,\, \Big\downarrow \iota_{x,i}\\
&& E & \stackrel{D}{\longrightarrow} & E\otimes V^*\\
&&&& \,\, \Big\downarrow \gamma\\
0 & \longrightarrow & (E_x/E^i_x)\otimes \phi^*((K_X)_x) & \stackrel{\beta_1}{\longrightarrow} &
(E_x/E^i_x)\otimes V^*_x & \xrightarrow{\,\,\, {\rm Id}_{E_x/E^i_x}\otimes q_x\,\,} & (E_x/E^i_x)
\otimes {\mathcal Q}_x & \longrightarrow & 0
\end{matrix}
\end{equation}
where $\iota_{x,i}$ and $q_x$ are the homomorphisms in \eqref{e5} and \eqref{e3} respectively, while $\gamma$ is the tensor
product of the natural projection $V^*\, \longrightarrow\, V^*_x$ with the composition
of maps $E\, \longrightarrow\, E_x\, \longrightarrow\, (E_x/E^i_x)$, and $\beta_1$ is the tensor product of the
identity map of $E_x/E^i_x$ with the natural inclusion map $\phi^*((K_X)_x)\, \hookrightarrow\, V^*_x$; the row at the
bottom of \eqref{e9} is exact. The given condition that ${\mathcal S}_x$ maps $E^i_x$ to
$E^i_x\otimes {\mathcal Q}_x$ is equivalent to the condition that the composition of maps in \eqref{e9}
$$
({\rm Id}_{E_x/E^i_x}\otimes q_x)\circ \gamma\circ D\circ\iota_{x,i}\,\,:\,\,
{\mathcal E}_{x,i}\,\, \longrightarrow\,\, (E_x/E^i_x)\otimes {\mathcal Q}_x
$$
vanishes. Consequently, $\gamma\circ D\circ\iota_{x,i}$ produces a homomorphism
$$
{\mathbb D}_{x,i}\,:\, {\mathcal E}_{x,i}\, \longrightarrow\, (E_x/E^i_x)\otimes \phi^*((K_X)_x).
$$
{}From \eqref{e-4} we conclude that ${\mathbb D}_{x,i}(f\cdot s) \,=\, f(x)\cdot {\mathbb D}_{x,i}(s)$
for every holomorphic section $s$ of ${\mathcal E}_{x,i}$ defined around $x$ and every holomorphic function $f$
defined around $x\,\in\, X$ (note that the image of $s$ in $E_x/E^i_x$ vanishes).
\end{proof}

\begin{proposition}\label{prop1}
Take a Lie algebroid connection $D\,:\, E\,\longrightarrow \, E\otimes V^*$. It gives a quasi-parabolic
Lie algebroid connection on the
parabolic vector bundle $E_*$ if and only if the following two statements hold:
\begin{enumerate}
\item For every $x\, \in\, S$, the homomorphism ${\mathcal S}_x$ in \eqref{e4} maps $E^i_x\, \subset\, E_x$ to
$E^i_x\otimes {\mathcal Q}_x$ for all $1\, \leq\, i\, \leq\, \ell_x+1$.

\item For all $x\, \in\, S$, the homomorphism of fibers $\phi^*_x\, :\, (K_X)_x\, \longrightarrow\, V^*_x$
(see \eqref{e2}) is the zero map.
\end{enumerate}
\end{proposition}

\begin{proof}
First assume that $D$ gives a quasi-parabolic Lie algebroid connection on the
parabolic vector bundle $E_*$. Fix any $x\, \in\, S$ and $1\, \leq\, i\, \leq\, \ell_x+1$. Take any
$v\, \in\, E^i_x\, \subset\, E_x$, and choose a holomorphic section $s$ of $E$, defined on an open neighborhood $U$ of $x$,
such that $s(x)\,=\, v$. From \eqref{e5} it follows that $s\, \in\, H^0(U,\, {\mathcal E}_{x,i}\big\vert_U)$. From \eqref{e7}
we know that
$$
D(s)\,\, \subset\,\, H^0\left(U,\, ({\mathcal E}_{x,i}\otimes V^*)\big\vert_U \right).
$$
Therefore, $\widehat{s}$ in \eqref{ws} satisfies the condition $$\widehat{s}\, \in\, ({\mathcal E}_{x,i})_x\otimes {\mathcal Q}_x
\,=\, ({\mathcal E}_{x,i}\otimes {\mathcal Q})_x.$$
Now from the construction of map ${\mathcal S}_x$ in \eqref{e4} it follows immediately that
${\mathcal S}_x(v)\, \in\, E^i_x\otimes {\mathcal Q}_x$.

To prove the second statement in the proposition, take any $x\, \in\, S$ and any $w\, \in\, E_x$. Since the first statement
in the proposition holds, Lemma \ref{lem1} says that
$$
{\mathbb D}_{x,i}
 \,\, :\,\, {\mathcal E}_{x,i}\,\, \longrightarrow\,\, (E_x/E^i_x)\otimes\phi^*((K_X)_x)
\,\,\subset\,\, (E_x/E^i_x)\otimes V^*_x.
$$
Furthermore, we have
\begin{equation}\label{e10}
{\mathbb D}_{x,i}\,\,=\,\, 0,
\end{equation}
because $D$ takes ${\mathcal E}_{x,i}$ to ${\mathcal E}_{x,i}\otimes V^*$. From \eqref{e10} we will deduce the
second statement that $\phi^*_x$ is the zero homomorphism.

For this, take any $v\, \in\, E_x\setminus E^i_x$ for some $i$ (note that $E_x\setminus E^i_x$ is nonempty when
$i\,=\, \ell_x+1$), and
choose a section $s\, \in\, H^0(U,\, E\big\vert_U)$ on a small open neighborhood $U$ of $x$ such that
$s(x)\,=\, v$. Fix a holomorphic function $f$ on $U$ such that $f(x)\,=\, 0$ and $(df)(x)\, \not=\,
0$. Then from \eqref{e5} it follows that $$f\cdot s \,\, \in\,\, H^0(U,\, {\mathcal E}_{x,i}\big\vert_U).$$
Now using \eqref{e-4} we have
\begin{equation}\label{e11}
D(f\cdot s)(x) \,=\, f(x)\cdot D(s)(x) + s(x)\otimes\phi^*(df)(x) \,=\, s(x)\otimes \phi^*(df)(x)
\end{equation}
because $f(x)\,=\, 0$. Let $\widehat{\gamma}\, :\, E_x\, \longrightarrow\, E_x/E^i_x$ be the quotient map. 
Note that ${\mathbb D}_{x,i}$ satisfies the condition $${\mathbb D}_{x,i}(f\cdot s)\,\,=\,\, (\widehat{\gamma}\otimes
{\rm Id}_{V^*_x})(D(f\cdot s)(x))$$
(see its construction in Lemma \ref{lem1}). Hence from \eqref{e11} it follows that
$$
{\mathbb D}_{x,i}(f\cdot s)\,\,=\,\,(\widehat{\gamma}\otimes {\rm Id}_{V^*_x})(D(f\cdot s)(x))
\,\,=\,\, \widehat{\gamma}(s(x))\otimes \phi^*(df)(x).
$$
Now \eqref{e10} implies that
\begin{equation}\label{e8}
\widehat{\gamma}(s(x))\otimes \phi^*(df)(x)\,\,=\,\,{\mathbb D}_{x,i}(f\cdot s)\,\,=\,\,0.
\end{equation}
Note that $\widehat{\gamma}(s(x))\,\not=\, 0$ because $s(x)\,=\, v\, \in\, E_x\setminus E^i_x$. Hence
from \eqref{e8} it follows that $\phi^*(df)(x)\,=\, 0$. Since $(df)(x)\, \not=\, 0$, we conclude that
the homomorphism $\phi^*_x\, :\, (K_X)_x\, \longrightarrow\, V^*_x$ is the zero map.
This proves the second statement in the proposition.

To prove the converse, assume that the two statements in the proposition hold. We need to
show that $D$ is a quasi-parabolic Lie algebroid connection on $E_*$.

Since the first statement holds, $D$ produces a homomorphism
$$
{\mathbb D}_{x,i} \,\, :\,\, {\mathcal E}_{x,i}
\,\, \longrightarrow\,\, (E_x/E^i_x)\otimes\phi^*((K_X)_x)
\,\,\subset\,\, (E_x/E^i_x)\otimes V^*_x
$$
(see Lemma \ref{lem1}). Now the second statement implies that ${\mathbb D}_{x,i}\,=\, 0$. Therefore, using 
\eqref{e9} it follows that $\gamma\circ D\circ\iota_{x,i}\,=\, 0$. Hence from \eqref{e5} we know that
$D$ maps ${\mathcal E}_{x,i}$ to ${\mathcal E}_{x,i}\otimes V^*$. This completes the proof.
\end{proof}

\begin{remark}
\label{rmk:Qx1}
Observe that condition (2) of Proposition \ref{prop1} only depends on the Lie algebroid $(V,\,\phi)$ and not on 
the Lie algebroid connection nor on the parabolic vector bundle. It implies that for a Lie algebroid $(V,\,\phi)$
to admit parabolic representations in the form of quasi-parabolic $(V,\,\phi)$--connections $(E_*,\,D)$, it is 
necessary that $\phi^*|_S\,=\,0$. We will further analyze this fact from the perspective of Simpson's theory of 
$\Lambda$-modules \cite{Si3} in Section \ref{para_lambda_mod}. Observe that, if we assume the necessary 
condition $\phi_x^*\,=\,0$, then we have $\mathcal{Q}_x\,=\,V_x$ for each parabolic point $x\,\in\, S$, so the
map $\mathcal{S}_x$ becomes a map $${\mathcal S}_x\,\,:\,\, E_x\,\, \longrightarrow\,\, E_x\otimes V_x.$$ over 
parabolic points.
\end{remark}

\subsection{Parabolic Lie algebroid connections}\label{para-conn}

Let $D$ be a quasi-parabolic Lie algebroid connection on a parabolic vector
bundle $E_*$. From Proposition \eqref{prop1}(2) we have that the homomorphism
of fibers $\phi^*_x\, :\, (K_X)_x\, \longrightarrow\, V^*_x$ is the zero map for every
$x \,\in\, S$. Consequently, the homomorphism $\phi^* \,:\, K_X \,\longrightarrow\, V^*$
factors through the natural homomorphism $K_X \,\hookrightarrow\, K_X \otimes {\mathcal O}_X(S)$,
where $S\,=\,
\sum_{i=1}^m x_i$ is the reduced effective divisor; in other words, we have a unique homomorphism 
\begin{equation}\label{e12}
\widetilde{\phi}^*\,\, : \,\, K_X \otimes {\mathcal O}_X(S) \,\,\longrightarrow\,\, V^*
\end{equation}
whose restriction to the subsheaf $K_X\, \subset\, K_X \otimes {\mathcal O}_X(S)$ coincides with $\phi^*$.

Note that for every $x\, \in\, S$, we have $(K_X \otimes {\mathcal O}_X(S))_x \,=\, \mathbb{C}$
(Poincar\'e adjunction
formula \cite[p.~146]{GH}). Consider the homomorphism $q(x)\, :\, V^*_x\, \longrightarrow\, {\mathcal Q}_x$ (see
\eqref{e3}). Composing it with $\widetilde{\phi}^*(x)$ in \eqref{e12}, we have
\begin{equation}\label{e13}
\psi_x \,:\, \mathbb{C}\,=\, (K_X \otimes {\mathcal O}_X(S))_x \, \longrightarrow\, {\mathcal Q}_x.
\end{equation}

{}From \eqref{e5} we have an exact sequence
\begin{equation}\label{z0}
0\, \longrightarrow\, {\mathcal E}_{x,i+1}\, \longrightarrow\, {\mathcal E}_{x,i} \, \longrightarrow\, E^i_x/E^{i+1}_x\,
\longrightarrow\, 0
\end{equation}
for all $x\, \in\, S$ and $1\, \leq\, i\, \leq\, \ell_x$. Since the quasi-parabolic Lie algebroid connection $D$
preserves ${\mathcal E}_{x,i}$ (see \eqref{e7}), from \eqref{z0} we have a commutative diagram
\begin{equation}\label{cd2}
\begin{matrix}
0 & \longrightarrow & {\mathcal E}_{x,i+1} & \longrightarrow & {\mathcal E}_{x,i} & \longrightarrow & E^i_x/E^{i+1}_x &
\longrightarrow & 0\\
&&\,\,\, \Big\downarrow D &&\,\,\, \Big\downarrow D &&\,\,\, \Big\downarrow \widehat{D}\\
0 & \longrightarrow & {\mathcal E}_{x,i+1}\otimes V^* & \longrightarrow &
{\mathcal E}_{x,i}\otimes V^* & \longrightarrow & (E^i_x/E^{i+1}_x)\otimes V^* &\longrightarrow & 0
\end{matrix}
\end{equation}
where $\widehat{D}$ is induced by $D$. Since $D$ satisfies \eqref{e-4}, we have
\begin{equation}\label{z-1}
\widehat{D}(fs) \,=\, f\cdot\widehat{D}(s) + s\otimes \phi^*(df)
\end{equation}
for any section $s$ of $E^i_x/E^{i+1}_x$ (note that
$E^i_x/E^{i+1}_x$ is a torsion sheaf supported at $x$) and any holomorphic function $f$ defined around $x\, \in\, X$.
Consider the composition of maps
\begin{equation}\label{z1}
({\rm Id}_{E^i_x/E^{i+1}_x}\times q)\circ \widehat{D}\,\,:\,\, E^i_x/E^{i+1}_x \,\,
\longrightarrow\,\, (E^i_x/E^{i+1}_x)\otimes {\mathcal Q}_x,
\end{equation}
where $q$ is the homomorphism in \eqref{e3}. From \eqref{z-1} it follows that
$({\rm Id}_{E^i_x/E^{i+1}_x}\times q)\circ \widehat{D}$
has the property $$({\rm Id}_{E^i_x/E^{i+1}_x}\times q)\circ \widehat{D}(fs) \,\,=\,\, f\cdot
({\rm Id}_{E^i_x/E^{i+1}_x}\times q)\circ\widehat{D}(s),$$ and hence
$({\rm Id}_{E^i_x/E^{i+1}_x}\times q)\circ \widehat{D}$
 produces a $\mathbb C$--linear homomorphism
\begin{equation}\label{z2}
\widehat{D}_{i,x}\,\,:\,\, E^i_x/E^{i+1}_x \,\, \longrightarrow\,\, (E^i_x/E^{i+1}_x)\otimes {\mathcal Q}_x
\end{equation}
of fibers.

\begin{definition}\label{def:paraLieconn}
The quasi-parabolic Lie algebroid connection $D$ will be called a \textit{parabolic Lie algebroid connection} if
the homomorphism $\widehat{D}_{i,x}$ in \eqref{z2} coincides with the homomorphism
$E^i_x/E^{i+1}_x \,\, \longrightarrow\,\, (E^i_x/E^{i+1}_x)\otimes {\mathcal Q}_x$ defined by $v\, \longmapsto\, v\otimes
(\psi_x(\alpha^x_i))$, where $\psi_x$ is the homomorphism in \eqref{e13} and $\alpha^x_i$ is the parabolic weight
in \eqref{ew}.
\end{definition}

\subsection{The Atiyah exact sequence} 

Let $(V, \, \phi)$ be a Lie algebroid on $X$. 
We first define generalized Lie algebroid connections on a vector bundle over an open subset of $X$.

Take an open subset $U\, \subset\, X$, and denote $S_U\, :=\, U \bigcap S$. Fix
a holomorphic function $w$ on $U$. The restriction $(V,\, \phi)\big\vert_U$
of the Lie algebroid $(V,\, \phi)$ to the open subset $U$ will be denoted by $(V_U,\, \phi_U)$.
A {\it generalized Lie algebroid connection with weight $w$} on a holomorphic vector
bundle ${\mathcal W}$ defined over $U$ is a holomorphic differential operator 
$$D \,\,:\,\, {\mathcal W} \,\,\longrightarrow \,\, {\mathcal W} \otimes V^*_U $$
satisfying the Leibniz rule
\begin{equation}\label{e15}
D(fs) \,\,=\,\, f D(s)\, +\, w \cdot s \otimes \phi^*_U (df),
\end{equation}
where $f$ is any locally defined holomorphic function on $U$ and $s$ is any locally
defined holomorphic section of $\mathcal W$.
Note that $D$ is a Lie algebroid connection on $\mathcal W$ if $w \,\equiv\, 1$.

A generalized Lie algebroid connection on a holomorphic vector 
$\mathcal W$ defined over $U$ is a pair $(w,\, D)$, where $w$ is a holomorphic function
on $U$ and $D$ is a generalized Lie algebroid connection on $\mathcal W$ with weight $w$. 

Take a generalized Lie algebroid connection $(w,\, D)$ on ${\mathcal W}\, \longrightarrow\, U$. Consider
the homomorphism
$$
({\rm Id}_{\mathcal W}\otimes q)\circ D\,\, :\,\, {\mathcal W}\,\, \longrightarrow\,\, {\mathcal W}\otimes
{\mathcal Q}\big\vert_U,
$$
where $q$ is the projection in \eqref{e3}. From \eqref{e15} it follows immediately that $({\rm 
Id}_{\mathcal W}\otimes q)\circ D (f\cdot s)\,=\, f\cdot q\circ D (s)$, where $s$ is any locally defined 
holomorphic section of $\mathcal W$ and $f$ is any locally defined holomorphic function on $U$. Therefore,
$({\rm Id}_{\mathcal W}\otimes q)\circ D$ produces a $\mathbb C$--linear homomorphism of fibers
\begin{equation}\label{e16}
{\mathcal S}_x\,\,:\,\, {\mathcal W}_x\, \,\longrightarrow\, \,{\mathcal W}_x\otimes {\mathcal Q}_x
\end{equation}
for every point $x \,\in\, U$.

Take a holomorphic vector bundle $E$ over $X$.
Let $\mathcal{C}_{E,V}$ denote the sheaf of generalized Lie algebroid connections on $E$ over $X$.
So the sections of $\mathcal{C}_{E,V}$ over an open set $U \,\subset\, X$ are the
generalized Lie algebroid connections on $E\big\vert_U$.
It is evident that $\mathcal{C}_{E,V}$ is a locally free coherent analytic sheaf fitting in the
short exact sequence
\begin{equation}\label{y1}
0\, \longrightarrow\, \text{End}(E)\otimes V^* \, \longrightarrow\, \mathcal{C}_{E,V}
\, \stackrel{\sigma}{\longrightarrow}\, {\mathcal O}_X \, \longrightarrow\, 0;
\end{equation}
the above projection $\sigma$ sends any locally defined generalized Lie algebroid connection $(w,\, D)$
on $E\big\vert_U$ to the holomorphic function $w$.

Let $E_*\,=\, (E,\, \{\{E^i_x\}_{i=1}^{\ell_x}\}_{x\in S},\, \{\{\alpha^x_i\}_{i=1}^{\ell_x}\}_{x\in S})$
be a parabolic vector bundle on $X$. Take a Lie algebroid $(V, \, \phi)$ on $X$ satisfying the following condition:
For every $x\, \in\, S$, the homomorphism of fibers $\phi^*_x\, :\, (K_X)_x\, \longrightarrow\,
V^*_x$ (see \eqref{e2}) is the zero map.

Take an open subset $U\, \subset\, X$. A \textit{generalized quasi-parabolic Lie algebroid connection} on $E_*$
over $U$ is a generalized Lie algebroid connection
$$
D\,\,:\,\, E\big\vert_U \,\,\longrightarrow\, \, (E\otimes V^*)\big\vert_U
$$
on $E\big\vert_U$ (see \eqref{e-4}) such that
\begin{equation}\label{q1}
D({\mathcal E}_{x,i}\big\vert_U) \,\, \subset\, \, ({\mathcal E}_{x,i}\otimes V^*)\big\vert_U
\end{equation}
for all $x\, \in\, S\bigcap U$ and every $1\, \leq\, i\, \leq\, \ell_x+1$ (see \eqref{e5}).

The following is a generalization of Proposition \ref{prop1}.

\begin{lemma}\label{lem3}
A generalized Lie algebroid connection $(w,\, D)$ on $E\big\vert_U$
gives a generalized quasi-parabolic Lie algebroid connection on the
parabolic vector bundle $E_*\big\vert_U$ if and only if the following holds:
For every $x\, \in\, S\bigcap U$, the homomorphism
$$
{\mathcal S}_x\,\,:\,\, E_x\, \,\longrightarrow\, \, E_x\otimes {\mathcal Q}_x
$$
in \eqref{e16} maps $E^i_x\, \subset\, E_x$ to
$E^i_x\otimes {\mathcal Q}_x$ for all $1\, \leq\, i\, \leq\, \ell_x+1$.
\end{lemma}

\begin{proof}
If $(w,\, D)$ is a generalized quasi-parabolic Lie algebroid connection, then for
any $x\, \in\, S\bigcap U$, from \eqref{q1}
it follows immediately that ${\mathcal S}_x(E^i_x)\, \subset\, 
E^i_x\otimes {\mathcal Q}_x$ for all $1\, \leq\, i\, \leq\, \ell_x+1$.

The proof of the converse statement is very similar to the proof in Proposition \ref{prop1}.
The details are omitted.
\end{proof}

Let $(w,\, D)$ be a generalized quasi-parabolic Lie algebroid connection on $E_*\big\vert_U$.
Then using \eqref{cd2} it follows that the homomorphisms ${\mathcal S}_x\big\vert_{E^i_x}$,
$1\,\leq\, i\, \leq\, \ell_x$ (see Lemma \ref{lem3}), produce homomorphisms
\begin{equation}\label{z3}
\widehat{D}_{i,x}\,\,:\,\, E^i_x/E^{i+1}_x \,\, \longrightarrow\,\, (E^i_x/E^{i+1}_x)\otimes {\mathcal Q}_x
\end{equation}
(see \eqref{z2}).

A \textit{generalized parabolic Lie algebroid connection} on $E_*\big\vert_U$ is
a generalized quasi-parabolic Lie algebroid connection $(w,\, D)$ on $E_*\big\vert_U$ such that
for every $x\, \in\, S\bigcap U$, and every $1\,\leq\, i\, \leq\, \ell_x$, the homomorphism
$\widehat{D}_{i,x}$ in \eqref{z3} coincides with the homomorphism
$$E^i_x/E^{i+1}_x \,\, \longrightarrow\,\, (E^i_x/E^{i+1}_x)\otimes {\mathcal Q}_x$$ defined by
$v\, \longmapsto\, w(x)\cdot v\otimes (\psi_x(\alpha^x_i))$, where $\psi_x$ is the homomorphism
in \eqref{e13} and $\alpha^x_i$ is the parabolic weight in \eqref{ew}.

The sheaf of generalized parabolic Lie algebroid connections on $E_*$ is evidently a subsheaf
of $\mathcal{C}_{E,V}$ (see \eqref{y1}). In fact, it is a coherent analytic subsheaf
of $\mathcal{C}_{E,V}$. The sheaf of generalized parabolic Lie algebroid connections on $E_*$
will be denoted by $\mathcal{C}_{E_*,V}$.

\begin{lemma}\label{lem4}
The sheaf of generalized parabolic Lie algebroid connections on $E_*$, namely $\mathcal{C}_{E_*,V}$,
fits in the following short exact sequence of holomorphic vector bundles on $X$:
\begin{equation}
\label{Atiya_para_Lie}
0\, \longrightarrow\, {\rm End}_n(E)\otimes V^* \, \longrightarrow\, \mathcal{C}_{E_*,V}
\, \stackrel{\sigma}{\longrightarrow}\, {\mathcal O}_X \, \longrightarrow\, 0,
\end{equation}

where ${\rm End}_n(E)$ is defined in \eqref{pe2}.
\end{lemma}

\begin{proof}
Restrict the homomorphism $\sigma$ in \eqref{y1} to the subsheaf
${\mathcal C}_{E_*,V}\, \subset\, {\mathcal C}_{E,V}$. The projection $\sigma$ in the lemma is
this restriction. Next note that if $D_1$ and $D_2$ are two
generalized quasi-parabolic Lie algebroid connections on $E_*\big\vert_U$ such that
$\sigma(D_1)\,=\, \sigma(D_2)$, then we have
$$D_1-D_2\, \in\, H^0(U,\, (\text{End}_P(E)\otimes V^*)\big\vert_U),$$ where $\text{End}_P(E)$ is defined
in \eqref{pe}. Furthermore, for any $s\, \in\, H^0(U,\, (\text{End}_P(E)\otimes V^*)\big\vert_U)$, it is
evident that $D_1+s$
is a generalized quasi-parabolic Lie algebroid connection on $E_*\big\vert_U$.

Now, by comparing the residues at the parabolic points, it is straightforward to check that if $D_1$ is a
generalized parabolic Lie algebroid connection on $E_*\big\vert_U$, then
$D_1+s$ is a generalized parabolic Lie algebroid connection on $E_*\big\vert_U$ if and only if
$s\, \in\, H^0\left(U,\, (\text{End}_n(E)\otimes V^*)\big\vert_U \right)$.
\end{proof}

{}From the definition of $\mathcal{C}_{E_*,V}$ it follows immediately that a
parabolic Lie algebroid connection on $E_*\big\vert_U$ is precisely a holomorphic
splitting of the short exact sequence in Lemma \ref{lem4} over $U\, \subset\, X$. More precisely, if
$$
\tau\, \, :\,\, {\mathcal O}_U\,\, \longrightarrow\,\, \mathcal{C}_{E_*,V}
$$
is a holomorphic splitting of the short exact sequence in Lemma \ref{lem4} over $U\,
\subset\, X$ (so $\sigma\circ\tau\,=\, {\rm Id}_{{\mathcal O}_U}$), then $\tau(1)$ is a parabolic Lie
algebroid connection on $E_*\big\vert_U$.
Conversely, if $D$ is a parabolic Lie algebroid connection on $E_*\big\vert_U$, then
we have a holomorphic splitting
$$
\tau'\, \, :\,\, {\mathcal O}_U\,\, \longrightarrow\,\, \mathcal{C}_{E_*,V}
$$
of the short exact sequence in Lemma \ref{lem4} over $U\, \subset\, X$ which is
uniquely determined by the condition that $\tau'(1)\,=\, D$.

The short exact sequence in Lemma \ref{lem4} will be called the \textit{Atiyah exact sequence}
of $E_*$ for the Lie algebroid $(V, \, \phi)$.

\section{A criterion for parabolic Lie algebroid connections}
\label{Crit_para_Lie_conn}
Let $(V,\, \phi)$ be a Lie algebroid on $X$ such that for all $x\, \in\, S$, the homomorphism of
fibers $\phi^*_x\, :\, (K_X)_x\, \longrightarrow\, V^*_x$ (see \eqref{e2}) is the zero map.

Consider the short exact sequence of holomorphic vector bundles on $X$ in Lemma \ref{lem4}. The
obstruction to its holomorphic splitting is given by a class
\begin{equation}\label{ze}
\zeta\,\, \in\, \, H^1(X,\, \text{Hom}({\mathcal O}_X,\, {\rm End}_n(E)\otimes V^*))
\,\,=\,\, H^1(X,\, {\rm End}_n(E)\otimes V^*).
\end{equation}

We have ${\rm End}_n(E)^*\,=\, {\rm End}_P(E)\otimes {\mathcal O}_X(S)$, where ${\rm End}_P(E)$ is
defined in \eqref{pe}. Therefore, using Serre duality,
\begin{equation}\label{z2n}
\zeta\,\, \in\,\, H^0(X,\, {\rm End}_P(E)\otimes V\otimes K_X)^*
\end{equation}
(see \eqref{ze}). In other words, for every parabolic vector bundle $E_*$ there is a naturally
associated homomorphism
\begin{equation}\label{z3n}
\zeta\, \,:\,\, H^0(X,\, {\rm End}_P(E)\otimes V\otimes K_X) \,\, \longrightarrow\,\, {\mathbb C}.
\end{equation}

The following proposition is a parabolic analog of \cite[Proposition 3.1]{BKS}.

\begin{proposition}\label{prop2}
Assume that the following two conditions hold:
\begin{enumerate}
\item The vector bundle $V$ is stable, and

\item ${\rm rank}(V)\,\, \geq\,\, 2$.
\end{enumerate}
Then any parabolic vector bundle $E_*$ admits a parabolic Lie algebroid connection.
\end{proposition}

\begin{proof}
First assume that $$\text{degree}(V) \, \geq\, {\rm rank}(V)\cdot (\text{degree}(TX)-\# S)\,=\,
{\rm rank}(V)\cdot (\text{degree}(TX)-m).$$
Since
$V$ is a stable vector bundle of rank at least two, from this we conclude that there is no nonzero
homomorphism from $V$ to $TX\otimes {\mathcal O}_X(-S)$, or in other words, there is no nonzero
homomorphism from $V$ to $TX$ that vanishes over $S$. In particular, the homomorphism $\phi$ is
the zero map (note that $\phi$ vanishes over $S$ because $\phi^*$ vanishes over $S$). Therefore, a parabolic
Lie algebroid connection on $E_*$ is simply
a strongly parabolic Higgs field on $E_*$. In particular, $D\,\equiv\, 0$ is a parabolic Lie algebroid
connection on $E_*$. Hence $E_*$ admits a parabolic Lie algebroid connection.

Next assume that the following two conditions hold:
\begin{enumerate}
\item $\text{degree}(V) \,< \, {\rm rank}(V)\cdot (\text{degree}(TX)-m)$, and

\item the parabolic vector bundle $E_*$ is parabolic semistable.
\end{enumerate}

{}From the above inequality $\text{degree}(V) \,< \, {\rm rank}(V)\cdot (\text{degree}(TX)-m)$ it follows
immediately that
\begin{equation}\label{z4}
\text{degree}(V\otimes K_X) \, \, <\,\, 0.
\end{equation}
Also, the vector bundle $V\otimes K_X$ is stable because $V$ is stable. Equip $V\otimes K_X$ with the
trivial parabolic structure, meaning it has no nonzero parabolic weights. The parabolic vector
bundle obtained this way will be denoted by $(V\otimes K_X)_*$. Note that $(V\otimes K_X)_*$ is parabolic
stable because $V\otimes K_X$ is stable and $(V\otimes K_X)_*$ has no nonzero parabolic weights.

Since $E_*$ is parabolic semistable, and $(V\otimes K_X)_*$ is parabolic stable, it follows that
the parabolic tensor product $E_*\otimes (V\otimes K_X)_*$ is parabolic semistable. We have
$$
\text{par-}\mu(E_*\otimes (V\otimes K_X)_*)\,\,=\,\, \text{par-}\mu(E_*) + \text{par-}\mu((V\otimes K_X)_*)
$$
$$
=\,\,
\text{par-}\mu(E_*) +\mu(V\otimes K_X)\,\, <\,\, \text{par-}\mu(E_*)
$$
(see \eqref{z4}). Since both $E_*\otimes (V\otimes K_X)_*$ and $E_*$ are parabolic semistable, this
implies that there is no nonzero parabolic homomorphism from $E_*$ to $E_*\otimes (V\otimes K_X)_*$.
In other words, we have
$$
H^0(X,\, {\rm End}_P(E)\otimes V\otimes K_X)\,=\, H^0(X,\, {\rm Hom}(E_*,\, E_*\otimes (V\otimes K_X)_*))
\,=\, 0.
$$
Consequently, the homomorphism $\zeta$ in \eqref{z3n} is the zero map. So $\zeta$ in \eqref{z2n}
satisfies the condition that $\zeta\,=\, 0$, and hence $\zeta$ in \eqref{ze} vanishes. 
This means that the short exact sequence of holomorphic vector bundles on $X$ in Lemma \ref{lem4} splits
holomorphically. Therefore, $E_*$ admits a parabolic Lie algebroid connection.

Finally, assume that the following two conditions hold:
\begin{enumerate}
\item $\text{degree}(V) \,< \, {\rm rank}(V)\cdot (\text{degree}(TX)-m)$, and

\item the parabolic vector bundle $E_*$ is \textit{ not } parabolic semistable.
\end{enumerate}

Let
\begin{equation}\label{y1n}
0\, =\, F^0_* \, \subsetneq\, F^1_* \, \subsetneq\, F^2_* \, \subsetneq\, \cdots
\, \subsetneq\, F^{n-1}_* \, \subsetneq\, F^n_*\,=\, E_*
\end{equation}
be the Harder--Narasimhan filtration of the parabolic vector bundle $E_*$.

Take any
\begin{equation}\label{y2}
\theta \,\, \in\,\, H^0(X,\, {\rm End}_P(E)\otimes V\otimes K_X).
\end{equation}
We have
\begin{equation}\label{y3}
\theta(F^i_*)\, \subset\, F^{i-1}_* \otimes V\otimes K_X
\end{equation}
for all $1\, \leq\, i\, \leq\, n$ (see \eqref{y1n}) which
is deduced by simply repeating the argument in the proof of \cite[Proposition 3.1]{BKS} and
using \eqref{z4}.

Now consider $\mathcal{C}_{E_*,V}$ in Lemma \ref{lem4}. Let
$$
\mathcal{C}^F_{E_*,V}\,\, \subset\,\, \mathcal{C}_{E_*,V}
$$
be the subsheaf defined by the following condition: A holomorphic section $$s\, \in\,
H^0(U,\, \mathcal{C}_{E_*,V}\big\vert_U),$$ defined over an open subset $U\, \subset\, X$
is a section of $\mathcal{C}^F_{E_*,V}\big\vert_U$ if $s(F^i) \, \subset\, (F^i\otimes V^*)\big\vert_U$
for all $1\, \leq\, i\, \leq\, n$ (see \eqref{y1n}). Similarly, let
$$
{\rm End}^F_n(E)\,\, \subset\,\, {\rm End}_n(E)
$$
be the subsheaf defined by the following condition: A holomorphic section $$s\, \in\,
H^0(U,\, {\rm End}_n(E)\big\vert_U),$$ defined over an open subset $U\, \subset\, X$
is a section of ${\rm End}^F_n(E)\big\vert_U$ if $s(F^i) \, \subset\, F^i\big\vert_U$
for all $1\, \leq\, i\, \leq\, n$. Note that $\mathcal{C}^F_{E_*,V}$ (respectively, ${\rm End}^F_n(E)$)
is a holomorphic subbundle of $\mathcal{C}_{E_*,V}$ (respectively, ${\rm End}_n(E)$).
Then we have a commutative diagram
\begin{equation}\label{y4}
\begin{matrix}
0 & \longrightarrow & {\rm End}^F_n(E)\otimes V^* & \longrightarrow & \mathcal{C}^F_{E_*,V}
& \stackrel{\sigma}{\longrightarrow} & {\mathcal O}_X & \longrightarrow & 0\\
&&\,\,\,\Big\downarrow\beta &&\Big\downarrow &&\Big\Vert\\
0 & \longrightarrow & {\rm End}_n(E)\otimes V^* & \longrightarrow & \mathcal{C}_{E_*,V}
& \stackrel{\sigma}{\longrightarrow} & {\mathcal O}_X & \longrightarrow & 0
\end{matrix}
\end{equation}
where the vertical maps are the natural inclusion maps, the rows are exact, and the bottom exact sequence
is the one in Lemma \ref{lem4}.

{}From \eqref{y4} it follows that $\zeta$ in \eqref{ze} lies in the image of the homomorphism
$$
\beta_*\,\, :\,\, H^1(X,\, {\rm End}^F_n(E)\otimes V^*)\,\,
\longrightarrow\,\, H^1(X,\, {\rm End}_n(E)\otimes V^*)
$$
induced by $\beta$ in \eqref{y4}. More precisely,
\begin{equation}\label{t1}
\beta_*(\zeta^F)\,\,=\,\, \zeta,
\end{equation}
where $\zeta^F\, \in\, H^1(X,\, {\rm End}^F_n(E)\otimes V^*)$ is the extension class for the top
short exact sequence in \eqref{y4}.

Now using \eqref{y3} and \eqref{t1} it can be deduced that $\zeta(\theta)\,=\, 0$, where $\zeta$ is the homomorphism in
\eqref{z3n}. Indeed, this follows by simply repeating the argument in the proof
of \cite[Proposition 3.1]{BKS}. Hence $\zeta$ in \eqref{ze} vanishes. This implies that
the short exact sequence of holomorphic vector bundles on $X$ in Lemma \ref{lem4} splits
holomorphically, and consequently $E_*$ admits a parabolic Lie algebroid connection.
\end{proof}

The following proposition is a parabolic analog of \cite[Lemma 3.2]{BKS}.

\begin{proposition}\label{prop3}
Assume that ${\rm rank}(V)\,\, =\,\, 1$. Let $E_*$ be a parabolic vector bundle on $X$.
Then the following two statements hold:
\begin{enumerate}
\item If $\widetilde{\phi}^*\,\, : \,\, K_X \otimes {\mathcal O}_X(S) \,\,\longrightarrow\,\, V^*$ in \eqref{e12}
is not an isomorphism, then $E_*$ admits a parabolic Lie algebroid connection.

\item If $\widetilde{\phi}^*$ in \eqref{e12} is an isomorphism, then $E_*$ admits a parabolic Lie
algebroid connection if and only if the parabolic degree of every parabolic direct summand of $E_*$ is zero.
\end{enumerate}
\end{proposition}

\begin{proof}
If $\widetilde{\phi}^*$ in \eqref{e12} is an isomorphism, then a parabolic Lie
algebroid connection on $E_*$ is a holomorphic connection on the parabolic bundle $E_*$.
Therefore, the main result of \cite{BL} is the second statement of the proposition.

To prove the first statement, assume that $\widetilde{\phi}^*$ in \eqref{e12} is not an isomorphism.

First consider the case where $\text{degree}(V)\, \geq\, 2(1-\text{genus}(X))-m$. Note
that
$$
H^0(X,\, \text{Hom}(K_X\otimes{\mathcal O}_X(S),\, V^*)\,=\, 0
$$
if $\text{degree}(V)\, > \, 2(1-\text{genus}(X))-m$. Also, any nonzero homomorphism
$K_X\otimes{\mathcal O}_X(S)\,\longrightarrow\, V^*$ is an isomorphism if
$\text{degree}(V)\, = \, 2(1-\text{genus}(X))-m$. So we conclude that $\widetilde{\phi}^*\,=\, 0$.
Hence any strongly parabolic Higgs field on $E_*$, in particular the zero Higgs field, is
a parabolic Lie algebroid connection on $E_*$.

Now assume that
\begin{equation}\label{z5}
\text{degree}(V)\, < \, 2(1-\text{genus}(X))-m.
\end{equation}
We also assume that $\phi\,\,\not=\,\, 0$, because if $\phi\,=\, 0$, then, as before, $E_*$ has a
parabolic Lie algebroid connection given by the zero Higgs field.

Take any
$$
\theta \,\, \in\,\, H^0(X,\, {\rm End}_P(E)\otimes V\otimes K_X).
$$
Let
\begin{equation}\label{z6}
0\, =\, F^0_* \, \subsetneq\, F^1_* \, \subsetneq\, F^2_* \, \subsetneq\, \cdots
\, \subsetneq\, F^{n-1}_* \, \subsetneq\, F^n_*\,=\, E_*
\end{equation}
be the Harder--Narasimhan filtration of the parabolic vector bundle $E_*$. From \eqref{z5} we know that
$\text{degree}(V\otimes K_X)\, <\, 0$. Hence $\theta$ must be nilpotent with respect to the filtration
in \eqref{z6}, meaning $\theta(F^i_*)\, \subset\, F^{i-1}_*\otimes V\otimes K_X$ for all $1\, \leq\, i
\, \leq\, n$. Now, as in the proof of Proposition \ref{prop2},
this implies that $\zeta(\theta)\,=\, 0$, where $\zeta$ is the homomorphism in
\eqref{z3n}. Hence $\zeta$ in \eqref{ze} vanishes, and thus
the short exact sequence of holomorphic vector bundles on $X$ in Lemma \ref{lem4} splits
holomorphically. This implies that $E_*$ admits a parabolic Lie algebroid connection.
\end{proof}

Combining Proposition \ref{prop2} and Proposition \ref{prop3} we have the following theorem:

\begin{theorem}\label{thm1}
Assume that the vector bundle $V$ is stable. Let $E_*$ be a parabolic vector bundle on $X$.
Then the following two statements hold:
\begin{enumerate}
\item If $\widetilde{\phi}^*$ in \eqref{e12} is not an isomorphism, then $E_*$ admits a parabolic
Lie algebroid connection.

\item If $\widetilde{\phi}^*$ is an isomorphism, then $E_*$ admits a parabolic Lie algebroid connection
if and only if the parabolic degree of every parabolic direct summand of $E_*$ is zero.
\end{enumerate}
\end{theorem}

\section{Integrable parabolic connections and parabolic $\Lambda$-modules}
\label{para_lambda_mod}

Any Lie algebroid $(V,\,\phi)$ has an associated graded complex $(\bigwedge^\bullet V^* ,\, d_V)$ called
the Chevalley--Eilenberg--de Rham complex of $(V,\,\phi)$. In degree 0, the map $d_V^0\,:\, \mathcal{O}_X
\,\longrightarrow \, V^*$ is $\phi^*\circ d$ and, for any locally defined
section $\omega \,\in \,\bigwedge^nV^*$ with $n\,>\,0$, the differential
$d_V(\omega)$ is constructed as follows: Given local sections $v_1,\,\cdots,\,
v_{n+1}$ of $V$, define
$$
d_V(\omega)(v_1,\,\cdots,\,v_{n+1})\,\,=\,\,\sum_{i=1}^n (-1)^{i+1} \phi(v_i)(\omega(v_1,\,\cdots,
\,\widehat{v}_i,\,\cdots,\, v_{n+1}))
$$
$$
+\, \sum_{1\le i < j \le n+1} (-1)^{i+j} \omega([v_i,\, v_j],\, v_1,\,\cdots,\, \widehat{v}_i,\, \cdots,
\,\widehat{v}_j,\, \ldots , \, v_{n+1}). 
$$
In terms of this complex, a $(V,\,\phi)$ connection on $E$ can be described as a map $D\,:\,E\,\longrightarrow\,
E\otimes V^*$ such that
$$D(fs)\,\,=\,\,fD(s)+s\otimes d_V(f).$$
Any $(V,\,\phi)$--connection $D$ on $E$ extends naturally to a map $$D\,:\, E\otimes
\bigwedge\nolimits^\bullet V^* \,\longrightarrow\, E\otimes \bigwedge\nolimits^{\bullet +1} V^*$$
using the rule
$$D(s\otimes \omega)\,\,=\,\,D(s)\wedge\omega + s\otimes d_V(\omega).$$
We say that the connection $D$ is integrable (or flat) if
$$D\circ D\,\,=\,\,0.$$
By \cite[\S~2.3]{To2}, this is equivalent to the following. For each $v\,\in\, V$, let
$D_v \,\in\, \operatorname{End}_\mathbb{C}(E)$ be the $\mathbb{C}$--linear endomorphism of
$E$ induced by contracting the connection $D$ with $v$. Then $D$ is integrable if and only if for each
$v,\,w\,\in\, V$ we have
\begin{equation}
\label{eq:integrability}
D_{[u,\,v]}\,\,=\,\,[D_u,\,D_v].
\end{equation}

Given a Lie algebroid $(V,\,\phi)$, let $\Lambda_{V}$ denote its universal developing algebra,
which is constructed as follows
(see, for instance, \cite{To2}). Let
$$U\,\,=
$$
$$
\otimes^\bullet( \mathcal{O}_X\oplus V) / \left \langle u\otimes v - v \otimes u - [u,\,v],
\, v\otimes f - f \otimes v - \phi(v)(f) \,\big\vert\,\, f\,\in\, \mathcal{O}_X, \, u,\,v
\,\in\, V \right \rangle ,$$
where $\bigotimes^\bullet W$ denotes the tensor algebra on $W$. Let $i\,:\,\mathcal{O}_X\oplus V \,
\hookrightarrow\, U$ be the canonical inclusion and let $U^\dagger\,\subset\, U$ be the subalgebra
generated by $i(\mathcal{O}_X\oplus V)$. Then
$$\Lambda_V\,\,=\,\,U^\dagger \, /\, \left \langle i(f,0)\cdot i(g,v)-i(fg,fv) \,\big\vert\,\, f,\,g
\,\in\, \mathcal{O}_X, v\,\in\, V\right\rangle.$$
Intuitively, $\Lambda_V$ is the sheaf of associative $\mathcal{O}_X$--algebras generated by $V$ with the
relations $v\cdot w-w\cdot v\,=\,[v,\,w]$ and $v\cdot f - f \cdot v \,=\,\phi(v)(f)$ for local holomorphic sections $v$ and $w$ of $V$
and locally defined holomorphic functions $f$. The natural grading on the tensor algebra
$\bigotimes^\bullet( \mathcal{O}_X\oplus V)$ induces a filtration on the algebra $\Lambda_V$
$$\Lambda_V^0\,\subset\, \Lambda_V^1\,\subset \,\Lambda_V^2\,\subset \,\cdots \,\subset\, \Lambda_V$$
such that $\Lambda_V^0 \,=\,\mathcal{O}_X$, $\Lambda^1_V\,=\,\mathcal{O}_X\oplus V$ and
$\Lambda_V^i \cdot \Lambda_V^j \,=\, \Lambda_V^{i+j}$ for all $i,\,j$ (see \cite[\S~4]{To2} and
\cite[Lemma 2.2]{Al}). The previous construction through $U^\dagger$ ensures
that the action of $\Lambda_V^0\,=\,\mathcal{O}_X$ on $\Lambda_V$ agrees with its
$\mathcal{O}_X$--module structure. This filtration gives the algebra $\Lambda_V$ the structure of a
split quasipolynomial sheaf of rings of differential operators (in the sense of \cite[\S~2]{Si3}) and, by \cite[Theorem 1.2]{To2},
there exists a bijective correspondence between the integrable $(V,\,\phi)$--connections $D\,:\,E
\,\longrightarrow\, E\otimes V^*$ on a vector bundle $E$ and the $\Lambda_{V}$-module structures
$$\widetilde{D}\,\,:\,\,\Lambda_{V} \otimes E \,\,\longrightarrow\,\, E$$ on $E$ such that
the $\mathcal{O}_X$--module structure induced by the $\Lambda_V$--module structure coincides with the
natural $\mathcal{O}_X$--module structure of $E$.

In this framework, where Lie algebroid connections are understood as actions of a sheaf of differential operators on the bundle, a general (possibly non-integrable) $(V,\,\phi)$-connection on $E$ is equivalent to a
map of $\mathcal{O}_X$--bimodules
$$\widetilde{D}\,\,:\,\, \Lambda_V^1 \otimes_{\mathcal{O}_X} E\,\,\longrightarrow\,\, E.$$
Notice that, as $\Lambda_V^1$ generates $\Lambda_V$, any $\Lambda_V$--module structure on a vector
bundle $E$ is determined by the action of $\Lambda_V^1$ on it. However, not any action of
$\Lambda_V^1$ on a vector bundle extends to an action of $\Lambda_V$. The obstruction to this extension
is precisely the integrability \eqref{eq:integrability} of the associated $(V,\,\phi)$--connection.

Taking into account these correspondences, the description and analysis of parabolic Lie algebroid 
connections carried out in the previous sections can then be studied alternatively in the framework of 
parabolic $\Lambda$-modules. For any sheaf of ring of differential operators $\Lambda$, a notion of 
parabolic $\Lambda$-module on a marked curve $(X,\,S)$ was defined in \cite{Al} as follows:

\begin{definition}[{\cite[Definition 2.5]{Al}}]
\label{def:parabolicLambdaModule}
A parabolic $\Lambda$--module on $(X,\,S)$ is a locally free $\Lambda$--module
$\widetilde{D}\,:\,\Lambda\otimes E\,\longrightarrow\, E$ over $X$ together with a parabolic structure on $E$ over $S$
$$E_x \,=\, E^1_x \,\supsetneq\, E^2_x \,\supsetneq\, \, \cdots\, \supsetneq\, E^{\ell_x}_x\,\supsetneq\, E^{\ell_x+1}_x \,=\, 0,$$
$$0 \,\leq\, \alpha^x_1 \,< \,\cdots\, <\, \alpha^x_{\ell_x} \,< \,1,$$
such that its associated decreasing sequences of subsheaves of $E$ (see \eqref{e6} and \eqref{eq:parSubsheaves})
$$E\,=\,{\mathcal E}_{x,1}\,\supsetneq \,{\mathcal E}_{x,1}
\,\supsetneq \, \cdots \,\supsetneq\, {\mathcal E}_{x,l_x} \,\supsetneq\, {\mathcal E}_{x,l_x+1}\,=\,E(-x)$$
satisfy the condition that the image of $\Lambda \otimes \mathcal{E}_{x,i}$ under the morphism
$\widetilde{D}\,:\,\Lambda \otimes E \,\longrightarrow\, E$ lies in $\mathcal{E}_{x,i}$ for all $i\,=\,
1,\,\cdots,\, l_x+1$ and all $x\,\in\, S$.
\end{definition}

It is clear, from the construction, that having a quasi-parabolic $(V,\, \phi)$-connection $D\,:\,E\,\longrightarrow\, E\otimes V^*$ on a parabolic
vector bundle $E_*$ in the sense of Definition \ref{def:parabolicConnection} is equivalent to having a map of $\mathcal{O}_X$-bimodules
$$\widetilde{D}\,:\, \Lambda^1_V\otimes_{\mathcal{O}_X} E \,\longrightarrow \,E$$
such that $\widetilde{D}(\Lambda_V^1\otimes_{\mathcal{O}_X} \mathcal{E}_{x,i})\,\subset\,
\mathcal{E}_{x,i}$. Thus, an integrable quasi-parabolic $(V,\,\phi)$--connection in the sense of Definition \ref{def:parabolicConnection}
is equivalent to a parabolic $\Lambda_V$--module structure on $E_*$ in the sense of Definition \ref{def:parabolicLambdaModule}.

From this new framework we can see that the Lie algebroid $(V,\phi)$ must satisfy certain conditions for quasi-parabolic Lie algebroid connections to exist.

\begin{proposition}\label{prop5}
Let $(V,\,\phi)$ be a Lie algebroid. For quasi-parabolic $(V,\,\phi)$-connections to exist, it is necessary that $\phi\,
:\,V\,\longrightarrow\, TX$ factors through $TX\otimes {\mathcal O}_X(-S)$.
\end{proposition}

\begin{proof}
Assume that there exists a $(V,\,\phi)$-connection $\widetilde{D}\,:\,\Lambda_V^1\otimes_{\mathcal{O}_X} E \,\longrightarrow\,
E$ preserving a quasi-parabolic structure $E_*$ on $E$. In particular, it must satisfy, for each $x\,\in\,S$,
$$\widetilde{D}(\Lambda_V^1\otimes_{\mathcal{O}_X} E(-x)) \,\,\subseteq\,\, E(-x).$$

Let $v$ be a local section of $V$, and let $f$ be a local section of $\mathcal{O}_X(-x)$ while $s$ be
any local section of $E$.
Then, as $\widetilde{D}$ is a map of $\mathcal{O}_X$--bimodules, we must have
$$\widetilde{D}([v,\,f]\otimes s) \,\,=\,\, \widetilde{D}((v\cdot f - f\cdot v)\otimes s)\,\,=\,\, \widetilde{D}(v\otimes fs)-f\widetilde{D}(v\otimes s).$$
The first summand of the right hand side is the action of $v$ on $fs\,\in\, E(-S)$.
The homomorphism $\widetilde{D}$ preserves $E(-x)$ if and only if the
outcome also lies in $E(-x)$. The second summand of the right hand side also belongs to $E(-x)$, since $f$ is a local section of
$\mathcal{O}_X(-x)$. Thus, we have $\widetilde{D}([v,\,f]\otimes s)\,\in\, E(-x)$. However,
$$\widetilde{D}([v,\,f]\otimes s)\,\, =\,\, \widetilde{D}(\phi(v)(f) \otimes s)\,\,=\,\, \phi(v)(f) s$$
and the only way for this to belong to $E(-x)$ for each local section $s\,\in\, E$ is
to have $[v,f]=\phi(v)(f)\,\in\, \mathcal{O}_X(-x)$ for each $v$ and $f$. In other words, $\mathcal{O}_X(-x)$ must be a Lie ideal in $\Lambda^1_V$. As
this holds for all $f\,\in\, \mathcal{O}_X(-x)$, we must have $\phi(v)
\,\in\, TX(-x)$ and, therefore, $\phi$ must factor through $TX(-x)$ for each $x\,\in\, S$. This completes the proof.
\end{proof}

\begin{remark}
\label{rmk:Qx2}
Observe that $\phi$ factors through $TX(-S)$ if and only if the
homomorphism $\phi^*\,:\,K_X\,\longrightarrow\, V^*$ vanishes at
each $x\,\in\, S$. This precisely coincides with condition (2) of Proposition \ref{prop1}. Also observe
that, if this condition is satisfied, then the sheaf $\mathcal{Q}$ defined in \eqref{e2} always satisfies
$\mathcal{Q}_x\,=\,V_x^*$ for each $x\,\in \,S$.
\end{remark}

Moreover, the proof of Proposition \ref{prop5} also shows that if the anchor map $\phi$ factors through 
$TX(-S)$, then any action $\widetilde{D}\,:\,\Lambda_V^1\otimes E\,\longrightarrow \,E$ on a vector bundle 
$E$ must preserve the subsheaf $E(-S)$. As a consequence, the construction of the residue of a parabolic 
$\Lambda$--module at a parabolic point from \cite[\S~6]{Al} extends naturally to any action 
$\widetilde{D}\,:\,\Lambda_V^1\otimes_{\mathcal{O}_X} E\,\longrightarrow\, E$, even if there only exists an 
action of $\Lambda_V^1$ instead of the entire algebra $\Lambda_V$.

Let $x\,\in\, S$ be a parabolic point, and let $i_x\,:\,\{x\}\,\hookrightarrow\, X$ be the inclusion
map. Since $\widetilde{D}$ preserves $E(-x)$, we have the following commutative diagram of sheaves of
$(\mathcal{O}_X,\,\mathcal{O}_X)$--bimodules
\begin{eqnarray*}
\xymatrixcolsep{1.8pc}
\xymatrix{
& V \otimes_{\mathcal{O}_X} E(-x) \ar[r] \ar[d] & V\otimes_{\mathcal{O}_X} E \ar[r] \ar[d]^{\widetilde{D}} & V\otimes_{\mathcal{O}_X} (i_x)_*E_x \ar@{-->}[d] \ar[r] & 0\\
0 \ar[r] & E(-x) \ar[r] & E \ar[r]^-{\operatorname{ev}} & (i_x)_*E_x \ar[r] & 0
}
\end{eqnarray*}
inducing a morphism
$$\widetilde{D}_x \, : \,V\otimes_{\mathcal{O}_X} (i_x)_* E_x\, \longrightarrow\, (i_x)_*E_x.$$
Taking the pullback by $i_x$ yields a residue map
$$\widetilde{\mathcal{S}}_x\,:\, V_x\otimes E_x \,\longrightarrow\, E_x$$
and $\widetilde{D}$ is quasi-parabolic if and only if $\widetilde{\mathcal{S}}_x$ preserves the
filtration $\{E_x^i\}$ of $E$. Under the
identification of $\mathcal{Q}_x$ with $V_x^*$ (see Remark \ref{rmk:Qx1} and Remark \ref{rmk:Qx2}), this agrees with condition (1) of Proposition \ref{prop1}.

A quasi-parabolic connection is parabolic in the sense of Definition \ref{def:paraLieconn} if the induced map
$$\widetilde{\mathcal{S}}_x\,:\, V_x\otimes E_x^i/E_x^{i+1} \,\longrightarrow\, E_x^i/E_x^{i+1}$$
coincides with the map
\begin{equation}
\label{eq:residualCond}
v\otimes s \,\, \longmapsto \,\, \alpha_i^x \widetilde{\phi}_x (v) s
\end{equation}
for $v\,\in \,V_x$ and $s\,\in\, E_x^i / E_x^{i+1}$, where $\widetilde{\phi}_x$ is the restriction of $\phi$ at each point $x\,
\in\, S$ considered as a map to its image on $TX(-S)_x$ which, by Poincar\'e adjunction formula \cite[p.~146]{GH}, is isomorphic to $\mathbb{C}$:
$$\widetilde{\phi}_x \,\,:\,\, V \,\,\stackrel{\phi_x}{\longrightarrow}\,\, TX(-S)_x \cong \mathbb{C}.$$

The agreement between these two frameworks for studying parabolic Lie algebroid connections allows us to apply \cite[Theorem 5.8]{Al} to prove the existence of moduli spaces of semistable parabolic Lie algebroid connections, for the following natural notion of parabolic stability. A Lie algebroid connection $D\,:\,E\,\longrightarrow \,E\otimes V^*$ is called semistable if for each $F\,\subset\, E$ preserved by $D$ we have
$$
\frac{\deg(F)+\sum_{x\in S}\sum_{i=1}^{l_x} \alpha_i^x (\dim(E_x^i\cap F_x)-\dim(E_x^{i+1}\cap F_x))}{\operatorname{rk}(F)}
$$
$$
\le\, \frac{\deg(E)+\sum_{x\in S}\sum_{i=1}^{l_x} \alpha_i^x (\dim(E_x^i)-\dim(E_x^{i+1}))}{\operatorname{rk}(E)}.
$$

Observe that if $E_*$ is a semistable parabolic vector bundle, then $(E_*,D)$ is semistable for each parabolic Lie algebroid connection $D$ on $E_*$.

\begin{theorem}
\label{thm2}
Let $(V,\,\phi)$ be any Lie algebroid such that $\phi|_S\,=\,0$. For every system of weights
$\alpha\,=\,\{\{\alpha_i^x\}_{i=1}^{l_x}\}_{x\in S}$, every parabolic type $\overline{r}\,=\,
\{\{r_i^x\}_{i=1}^{l_x}\}_{x\in S}$ and every degree $d$, there exists a quasi-projective coarse moduli
space $\mathcal{M}_{(V,\phi)}(\alpha,\,\overline{r},\,d)$ of semistable integrable parabolic
$(V,\,\phi)$--connections $(E_*,\,D)$ on
$(X,\,S)$, with $D\,:\,E\,\longrightarrow\, E\otimes V^*$, $\deg(E)\,=\,d$,
$\dim(E_x^i/E_x^{i+1})\,=\,r_x^i$ and system of weights $\alpha$.

If $V$ is a line bundle, then this moduli space is nonempty if and only if either
\begin{itemize}
\item $(V,\,\phi)\,\,\ne \,\,(TX(-S),\, i\,:\,TX(-S)\,\hookrightarrow\, TX)$, or

\item $(V,\,\phi)\,\,=\,\,(TX(-S),\,i\,:\,TX(-S)\,\hookrightarrow\, TX)$ and
$$d+\sum_{x\in S} \sum_{i=1}^{l_x} \alpha_i^x r_i^x \,=\, 0.$$
\end{itemize}
If $V$ is a stable bundle, then for the moduli space to be nonempty it is sufficient to have either
\begin{itemize}
\item $\operatorname{im} \phi \,\ne\, TX(-S)$, or
\item $\operatorname{im} \phi \,=\, TX(-S)$ and 
$$d+\sum_{x\in S} \sum_{i=1}^{l_x} \alpha_i^x r_i^x \,=\, 0.$$
\end{itemize}
\end{theorem}

\begin{proof}
The existence of a quasi-projective coarse moduli space $\mathcal{M}$ of quasi-parabolic Lie 
algebroid connections follows directly from the previous discussion on the equivalence between 
the quasi-parabolic $(V,\,\phi)$-connections and the parabolic $\Lambda_V$-modules by applying 
\cite[Theorem 5.8]{Al}. The condition \eqref{eq:residualCond} of being a parabolic connection 
can be rephrased as a set of residual conditions in the sense of \cite[Definition 6.2]{Al} as 
follows. Let $\{v_1^x,\, \cdots ,\, v_r^x\}$ be a basis of $V_x$ for each $x\,\in\, S$. For 
each $j\,=\,1,\,\cdots,\,r$, let $\overline{R}_j$ be the following set of sections of 
$\Lambda_V|_x$
$$
\overline{R}_j=\left\{(-\alpha_i^x \widetilde{\phi}_x (v_j^x), v_j^x) \,\in\, \mathcal{O}_X|_x 
\oplus V_x \,\subset\, \Lambda_V|_x\right \}_{x\in S,\, i=1,\cdots, l_x}
$$
Then $\overline{R}_j$ constitutes a residual condition for $\Lambda_V$ over the parabolic 
points $S$ and a quasi-parabolic connection $(E_*,D)$ is $\overline{R}_j$ --residual in the 
sense of \cite[Definition 6.2]{Al} for each $j\,=\,1,\,\cdots,\,r$ if and only if 
\eqref{eq:residualCond} is satisfied for every $v_j^x\,\in\, V_x$. Since 
\eqref{eq:residualCond} is linear in $V$ and $\{v_1^x,\,\cdots,\, v_r^x\}$ is a basis of $V$, 
this is equivalent to have \eqref{eq:residualCond} for each $v\,\in\, V$ and, therefore, to 
have a parabolic Lie algebroid connection. By \cite[Theorem 6.3]{Al}, for each $j$, 
$\overline{R}_j$--residual parabolic $\Lambda_V$ connections form a closed subscheme 
$\mathcal{M}_j$ of the moduli space $\mathcal{M}$ of quasi-parabolic Lie algebroid 
connections. Thus, the intersection $\mathcal{M}_{(V,\phi)}(\alpha,\,\overline{r},\,d) 
\,:=\,\cap_j \mathcal{M}_j$ is a closed subscheme of $\mathcal{M}$ corepresenting the moduli 
of semistable integrable parabolic $(V,\,\phi)$--connections.

Let us analyze the non-emptiness conditions. Assume first that $V$ is a line bundle. First of all, observe that if $V$ is a line bundle, then each parabolic Lie algebroid connection $D:E\longrightarrow E\otimes V^*$ is automatically integrable, as we have $\wedge^2 V = 0$ and, therefore, the map
$$D\circ D \,:\, E\,\longrightarrow \,E\otimes \wedge^2 V^*$$
is always zero. Thus, the moduli space is nonempty if and only if there exists a semistable parabolic Lie algebroid connection.

For any curve $X$, the moduli space of semistable parabolic bundles on $X$ with degree $d$, parabolic 
system of weights $\alpha$ and parabolic type $\overline{r}$ is nonempty and, therefore, there exists at least 
one parabolically polystable vector bundle $E_*$ on $(X,\,S)$ for the system of weights $\alpha$. 

First suppose that $(V,\,\phi)\,\,\ne\,\, (TX(-S),\,i\,:\,TX(-S)\,\hookrightarrow\, TX)$. Then, by
Proposition \ref{prop3}, the parabolic vector bundle
$E_*$ admits a parabolic Lie algebroid connection $D\,:\, E\,\longrightarrow\, E\otimes V^*$. As $E_*$ 
is polystable, it is semistable and, therefore, $(E_*,\,D)$ is semistable, so it yields a point in 
$\mathcal{M}_{(V,\phi)}(\alpha,\,\overline{r},\,d)$.

Next assume that $(V,\,\phi)\,\,=\,\, (TX(-S),\,i\,:\,TX(-S)\,\hookrightarrow\, TX)$ and that 
$$d+\sum_{x\in S} \sum_{i=1}^{l_x} \alpha_i r_i \,\,=\,\, 0.$$ As $E_*$ was taken to be 
polystable, any direct summand of $E_*$ must have the same parabolic slope as that of $E_*$. 
Let $r$ be the rank of $E$. Since the parabolic vector bundle $E_*$ has parabolic slope
$$\frac{d+\sum_{x\in S} \sum_{i=1}^{l_x} \alpha_i^x r_i^x}{r} \ = \ 0,$$
each direct summand of $E_*$ must have parabolic slope $0$, and, therefore, parabolic degree $0$. Thus, $E_*$
satisfies the conditions of part (2) of Proposition \ref{prop3} and, hence there exists a parabolic Lie
algebroid connection on $E_*$, which is, again due to the semistability of $E_*$, a semistable parabolic
Lie algebroid connection.

To show the converse, let us suppose that $(E_*,\,D)$ is a semistable parabolic Lie algebroid for
the Lie algebroid $(TX(-S),\,i\,:\,TX(-S)\,\hookrightarrow\, TX)$. Then by \cite{BL}, each direct
summand of
$E_*$ must have parabolic degree $0$, and therefore, the parabolic degree of $E_*$ must be zero, so
$$d+\sum_{x\in S} \sum_{i=1}^{l_x} \alpha_i^x r_i^x \,=\, 0.$$
This completes the proof for line bundles. Now, assume that $V$ is stable and $\operatorname{rk}(V)\ge 2$. Let
$L\,=\,\operatorname{im} \phi \subseteq TX(-S)$. If $L\,=\,0$, then a Lie algebroid connection for $(V,\phi)$ is simply an
$\mathcal{O}_X$--linear map $D\,:\,E\,\longrightarrow\, E\otimes V^*$. Thus, the map $D=0$ then constitutes a parabolic
integrable Lie algebroid connection on $E_*$. Suppose that $L\ne 0$. Then $L$ is a line bundle. The previous argument on the
line bundle case and the hypothesis
of the theorem imply that the polystable bundle $E_*$ admits a $(L,\,i:L\hookrightarrow TX(-S))$--parabolic connection
$$\nabla\, : \, E\,\longrightarrow\, E\otimes L^*.$$
By composition with $\operatorname{id}_E\otimes \phi^*$, we obtain then a quasi-parabolic Lie algebroid connection
$$D\, : \, E\,\longrightarrow\, E \otimes L^* \,\stackrel{\operatorname{id}\otimes \phi^*}{\longrightarrow}\, E\otimes V^*.$$
If we call $D_v\,:\, E\,\longrightarrow\, E$ the induced map by contracting with $v\,\in\, V$, then we have, by construction
$$D_v\, =\,\nabla_{\phi(v)}.$$
Observe that $D$ is a Lie algebroid connection for $(V,\phi)$, as, for each local sections $v$ in $V$ and$s$ in $E$ and each local holomorphic function $f$ we have
$$D_v(fs) \, =\,\nabla_{\phi(v)}(fs)\, =\, f\nabla_{\phi(v)}(s)+s\otimes \phi(v)(f) \, =\, fD_v(s)+s\otimes \phi(v)(f).$$
As $\nabla$ is parabolic, it satisfies \eqref{eq:residualCond} and, by construction, $D$ clearly also satisfies \eqref{eq:residualCond}, so it is also parabolic.
On the other hand, since $L\subset T_X(-S)$ is a line bundle, then $\nabla$ is integrable. Therefore, for each pair of local sections $u$ and $v$ of $V$ we have the following.
$$D_{[u,v]}\, =\, \nabla_{\phi([u,v])}\, =\, \nabla_{[\phi(u),\phi(v)]} \, =\,[\nabla_{\phi(u)},\,\nabla_{\phi(v)}] \, =\,[D_u,\,D_v]$$
so $D$ is integrable. As before, since $E_*$ is polystable, then $(E_*,D)$ is semistable and, therefore, it represents a point in the moduli space.
\end{proof}

\section*{Acknowledgments}
 
We thank the referee for a helpful suggestion.
D.A. was supported by grants PID2022-142024NB-I00 and RED2022-134463-T funded by 
MCIN/AEI/10.13039/501100011033. A.S. is partially supported by SERB SRG Grant SRG/2023/001006.
I.B. is partially supported by a J. C. Bose Fellowship (JBR/2023/000003).

\end{document}